\documentclass[12pt]{article}

\usepackage{
   amsthm,
   amssymb,
   graphicx,
   hyperref,
  mathtools,
  fullpage
   }
\usepackage[shortlabels]{enumitem}
\usepackage{stmaryrd}
\usepackage{tikz-cd}
\usepackage{tikz}

\theoremstyle{plain}
\newtheorem{thm}{Theorem}[section]
\newtheorem{lemma}[thm]{Lemma}
\newtheorem{propn}[thm]{Proposition}
\newtheorem{cor}[thm]{Corollary}
\theoremstyle{definition}
\newtheorem{defn}[thm]{Definition}
\newtheorem{rmk}[thm]{Remark}
\newtheorem{ex}[thm]{Example}
\newtheorem{claim}[thm]{Claim}

\newcommand{\R}{\mathbb{R}}
\newcommand{\C}{\mathbb{C}}
\newcommand{\id}{\mathrm{id}}
\newcommand{\bT}{^bT}
\newcommand{\bd}{{^bd}}
\newcommand{\bdel}{{^b\partial}}
\newcommand{\bdelbar}{{^b{\overline\partial}}}
\newcommand{\X}{\mathfrak{X}}
\newcommand{\bX}{{^b\mathfrak{X}}}
\newcommand{\bO}{{^b\mathcal{O}}}
\renewcommand{\Re}{\operatorname{Re}}

\newcommand{\an}{\rho}
\newcommand{\<}{\langle}
\renewcommand{\>}{\rangle}
\renewcommand{\L}{\mathcal{L}}
\newcommand{\Arg}{\operatorname{Arg}}
\newcommand{\V}{\mathcal{V}}

\numberwithin{equation}{section}

\title{The Newlander-Nirenberg theorem for \\ complex \texorpdfstring{$b$}{b}-manifolds}
\author{
Tatyana Barron and Michael Francis
}
\date{May 25, 2026}

\begin{document}

\maketitle

\begin{abstract}
\noindent Melrose defined the $b$-tangent bundle of a  smooth manifold $M$ with boundary  as the vector bundle  whose sections are vector fields on $M$ tangent to the boundary. Mendoza defined a complex $b$-manifold as a manifold with boundary  together with an involutive splitting of the  complexified $b$-tangent bundle into complex conjugate factors. We prove complex $b$-manifolds have a single local model depending only on dimension.  This can be thought of as the Newlander-Nirenberg theorem for complex $b$-manifolds. Our proof uses Mendoza's result that complex $b$-manifolds have no ``formal local invariants'' and a singular coordinate change to leverage the classical Newlander-Nirenberg theorem and Catlin's generalization for complex manifolds with pseudoconvex boundary.
\end{abstract}



\section{Introduction}

Throughout, we conflate  $\R^2 = \C$ and use coordinates $z=(x,y)=x+iy$  interchangeably. The  word ``smooth'' means ``infinitely real-differentiable'', not ``holomorphic''. We denote the closed half space in $\R^d$ by $\R^d_+ \coloneqq \R_+ \times \R^{d-1}$  where $\R_+ \coloneqq [0,\infty)$.

Recall that a complex manifold may be equivalently defined either by a holomorphic atlas or by an integrable almost-complex structure. The equivalence of these two definitions is  given by the Newlander-Nirenberg theorem \cite{NN}. In more  detail, a complex structure on a  smooth manifold $M$ is an involutive subbundle $T^{0,1}M$ of the complexified tangent bundle such that $\C TM = T^{1,0}M \oplus T^{0,1}M$ where $T^{1,0}M \coloneqq \overline{T^{0,1}M}$. The Newlander-Nirenberg theorem implies that every point belongs to a coordinate chart  $(x_1,y_1,\ldots,x_n,y_n)=(z_1,\ldots,z_n)$ in which  $T^{0,1}M$ is spanned by 
\begin{align*}
\partial_{\overline z_j} \coloneqq \tfrac{1}{2}(\partial_{x_j} +i \partial_{y_j}) && j=1,\ldots,n.
\end{align*}
Accordingly, $T^{1,0}M$ is spanned by 
\begin{align*}
\partial_{z_j} \coloneqq \tfrac{1}{2}(\partial_{x_j} -i \partial_{y_j}) && j=1,\ldots,n.
\end{align*}
 
Suppose now that $M$ is a smooth manifold with boundary. The  $b$-tangent bundle of $M$ is the vector bundle ${\bT M}$ over $M$ whose smooth sections are the vector fields on $M$ that are tangent along the boundary. These terminologies stem from the $b$-geometry of Melrose \cite{Melrose}, also called log-geometry in algebro-geometric contexts. Mendoza \cite{Mendoza} defined a  complex  $b$-structure   to be an involutive subbundle ${\bT^{0,1}}M$ of the complexified $b$-tangent bundle such that $\C{\bT M} = {\bT^{1,0}}M \oplus {\bT^{0,1}}M$, where ${\bT^{1,0}}M \coloneqq \overline{{\bT^{0,1}}M}$.   Thus,  complex $b$-structures are defined exactly analogously to complex structures, replacing the tangent bundle by the $b$-tangent bundle.

To complete the analogy, one would like there to be a Newlander-Nirenberg type result to the effect  that all complex $b$-structures are locally isomorphic to a standard model depending only on dimension. Such a result would empower one to give a second (equivalent) definition of a complex $b$-manifold in terms of an appropriately defined $b$-holomorphic atlas. We equip $\R^{2n+2}=\C^{n+1}$, $n \geq 0$  with coordinates $(x_0,y_0, \ldots,x_n,y_n)=(z_0,\ldots,z_n)$. 
A logical  candidate  model for a complex $b$-manifold is the $x_0 \geq 0$ half space $M=\R^{2n+2}_+$ with ${\bT^{0,1}}M$
spanned by:
\begin{align*}
\bdel_{\overline z_0} &\coloneqq \tfrac{1}{2}(x_0\partial_{x_0} +i \partial_{y_0}) &&   \\
\partial_{\overline z_j} &\coloneqq \tfrac{1}{2}(\partial_{x_j} +i \partial_{y_j}) && j=1,\ldots,n.
\end{align*}
Accordingly, ${\bT^{1,0}}M$ is spanned by: 
\begin{align*}
\bdel_{z_0} &\coloneqq \tfrac{1}{2}(x_0\partial_{x_0} -i \partial_{y_0}) &&  \\
\partial_{z_j} &\coloneqq \tfrac{1}{2}(\partial_{x_j} -i \partial_{y_j}) && j=1,\ldots,n.
\end{align*}

Our goal in this article is to show that, locally near points on the boundary, every complex $b$-manifold of dimension $2n+2$ is indeed isomorphic to the model example above. Mendoza \cite{Mendoza} already took up this problem and partially settled it, up to deformation by terms whose Taylor expansions vanish along the boundary. Indeed, Mendoza's result, which we  state below, will play a crucial role.

\begin{thm}[\cite{Mendoza}, Proposition~5.1]\label{Mendoza}
Let $M$ be a complex $b$-manifold of real dimension $2n+2$, $n \geq 0$. Then, every $p \in \partial M$ has an open neighbourhood $U \subseteq M$ on which there are smooth local coordinates $(x_0,y_0, \ldots,x_n,y_n)=(z_0,\ldots,z_n)$ centered at $p$ with $x_0$ vanishing on $\partial M$ and positive on the interior;
and smooth complex vector fields $\Gamma_0, \ldots,\Gamma_n$ on $U$ vanishing to infinite order on $\partial M$ such that
\begin{align*}
L_0  &\coloneqq \bdel_{\overline z_0} + \Gamma_0 \\
L_j &\coloneqq \partial_{\overline z_j} + \Gamma_j && j=1,\ldots,n 
\end{align*}
is a frame for ${\bT^{0,1}}M$ over $U$. Moreover, each of $\Gamma_0,\ldots,\Gamma_n$ is a linear combination of $\bdel_{z_0}, \partial_{z_1},\ldots,\partial_{z_n}$ whose coefficients are smooth complex-valued functions on $U$ vanishing to infinite order on $\partial M$. 
\end{thm}

The contribution of the present article will be to improve the above result by showing that the deformation terms $\Gamma_0,\ldots, \Gamma_n$ can be got rid of. Thus, our  main result is the following.

\begin{thm}\label{mainthm}
Let $M$ be complex $b$-manifold of real dimension $2n+2$, $n \geq 0$. Then, every $p \in \partial M$ has an open neighbourhood $U \subseteq M$ on which there are smooth local coordinates $(x_0,y_0, \ldots,x_n,y_n)=(z_0,\ldots,z_n)$ centered at $p$ with $x_0$ vanishing on $\partial M$ and positive on the interior such that
\begin{align*}
\bdel_{\overline z_0} &\coloneqq \tfrac{1}{2}(x_0\partial_{x_0}+i\partial_{y_0})\\
\partial_{\overline z_j} &\coloneqq \tfrac{1}{2}(\partial_{x_j}+i\partial_{y_j}) && j=1,\ldots,n 
\end{align*}
is a frame for ${\bT^{0,1}}M$ over $U$.
\end{thm}

Readers  familiar with the techniques needed to prove the classical Newlander-Nirenberg theorem will be unsurprised  that nonellipticity of the operator $\bdel_{\overline z} =\frac{1}{2}(x\partial_x + i \partial_y)$ is at the heart of the matter. The difficulty is that, without ellipticity,  we do not automatically have good existence theory for solutions to deformations of this operator. The main insight of this article is that it is possible to  get around this difficulty by leveraging the polar  coordinate change  $g(x,y)=xe^{iy}$. Notice that $\bdel_{\overline z} \, g =0$, i.e. $g$ is ``$b$-holomorphic''.

Let us  summarize the contents of this paper. In Section~\ref{sec:background}, we collect definitions and basic results on complex $b$-manifolds. In Section~\ref{sec:2exist}, we prove in the two-dimensional case that nontrivial $b$-holomorphic functions exist locally. In Section~\ref{sec:2main}, we use the latter existence result  to prove the main result Theorem~\ref{mainthm} in the two-dimensional case. In Section~\ref{new}, we recall Catlin's analog \cite{catlin} of the Newlander-Nirenberg theorem for manifolds with boundary and make preparations for extending our results to the higher-dimensional cases. Sections~\ref{sec:exist} and \ref{sec:main} extend the  results of Sections~\ref{sec:2exist} and \ref{sec:2main} to the general case. Finally, in Section~\ref{generalpush}, we situate the singular coordinate change $g(x,y)=xe^{iy}$ in  a more general context which will be relevant for planned future work on complex $b^k$-manifolds (see \cite{Barron-Francis} for the definition of a complex $b^k$-manifold).

\section{Background}\label{sec:background}

The language of $b$-geometry, also known as log-geometry, was introduced by Melrose \cite{Melrose}. One encounters two superficially different  formulations  of $b$-geometry  in the literature. In the  original approach, one works on a manifold with boundary. Other authors instead use a manifold without boundary that is equipped with a given hypersurface \cite{BLS}. We follow the former approach.

\begin{defn}
A \textbf{$\mathbf{b}$-vector field} on a manifold with boundary $M$ is a smooth  vector field on $M$ that is tangent along $\partial M$. The collection of $b$-vector fields is denoted $\bX(M) \subseteq \X(M)$.
\end{defn}

\begin{ex}\label{bex}
The closed half space  $\R^{n+1}_+$
has $\bX(\R^{n+1}_+)=\langle x_0 \partial_{x_0},\partial_{x_1},\ldots,\partial_{x_n}\rangle$, the free $C^\infty(\R^{n+1}_+)$-module generated by $x_0 \partial_{x_0},\partial_{x_1},\ldots,\partial_{x_n}$.
\end{ex}

Of course, the above example above completely captures the local structure of $b$-vector fields. Accordingly, $\bX(M)$ is a projective $C^\infty(M)$-module closed under Lie bracket for any manifold with boundary. Applying Serre-Swan duality to the inclusion $\bX(M)\to \X(M)$, one has a corresponding  Lie algebroid  $\an : {\bT M} \to TM$ whose anchor map $\an$ induces (abusing notation) an identification $C^\infty(M;{\bT M})= \bX(M)$.

\begin{defn}
For a manifold with boundary $M$, the Lie algebroid ${\bT M}=({\bT M},\an,[\cdot,\cdot])$    satisfying $C^\infty(M;{\bT M})= \bX(M)$ described above is called the \textbf{$\mathbf{b}$-tangent bundle} of $M$.
\end{defn}
In Example~\ref{bex}, $x_0 \partial_{x_0},\partial_{x_1},\ldots,\partial_{x_n}$ is a global frame for ${\bT M}$. The anchor map $\an$ descends from evaluation of vector fields. Thus, over the interior, $\an$ is an isomorphism and, over the boundary, the kernel of $\an$ is the line bundle spanned by $x_0 \partial_{x_0}$. In general, for  any $b$-manifold $M$, the bundles   ${\bT M}$ and  $TM$ are canonically isomorphic over the interior  via $\an$.

If $\theta:M_1\to M_2$ is a diffeomorphism of manifolds with boundary, it is clear that $\theta$ preserves the associated modules of $b$-vector fields. By Serre-Swan duality, $\theta$ induces a Lie algebroid isomorphism: 
\begin{align}\label{bpush}
{^b \theta}_* : {\bT M_1} \to {\bT M_2}.
\end{align}
Over the interiors, if we apply the aforementioned natural identifications of the $b$-tangent bundle and usual tangent bundle, ${^b\theta_*}$ coincides with  $\theta_*:T M_1 \to TM_2$, the usual   induced isomorphism given by pushforward of tangent vectors.

\begin{defn}
The \textbf{$\mathbf{b}$-cotangent bundle} ${\bT M^*}$ of a manifold with boundary $M$ is the dual of the $b$-tangent bundle ${\bT M}$.
\end{defn}

The dual $\an^* : TM^* \to {\bT M^*}$ of the anchor map $\an:{\bT M}\to TM$ is likewise an isomorphism when restricted over the interior. In particular, $\an^*$ is injective on a dense open set, and so (what is equivalent) induces an injective mapping $C^\infty(M;T M^*) \to C^\infty(M;{\bT M^*})$. Put in simpler terms, a 1-form $\omega \in C^\infty(M;T M^*)$   is fully determined by its   pairing with $b$-vector fields. The \textbf{$\mathbf{b}$-exterior derivative} of a smooth function $f$ on $M$ is defined as:
\begin{align}\label{bd} 
\bd f\coloneqq \an^*( df). 
\end{align}
Actually, one may quite legitimately regard $d f$  and $\bd f$ as identical, with the latter notation merely hinting that  one  only intends to pair the form with $b$-vector fields.

A general philosophy of $b$-calculus is that many classical geometries admit $b$-analogues in which  the role of the tangent bundle is played by  the $b$-tangent bundle. A notable example is $b$-symplectic geometry  (\cite{GMP}, \cite{LLSS}). Complex $b$-geometry was introduced by  Mendoza \cite{Mendoza} who furthermore made a systematic study of the $b$-Dolbeault complex. We remark that the basic idea of a complex $b$-structure is mentioned in passing on pp.~218 of \cite{Melrose}.

\begin{defn}
A \textbf{complex $\mathbf{b}$-structure} on an (even-dimensional)  manifold with boundary $M$  is a complex subbundle ${\bT^{0,1}}M$ of the complexified $b$-tangent bundle satisfying:
\begin{enumerate}[(i)]
\item $\C {\bT  M} =  {\bT^{1,0}} M  \oplus {\bT^{0,1}} M$, where ${\bT^{1,0}}M\coloneqq \overline{ {\bT^{0,1}} M }$,
\item ${\bT^{0,1}} M$ is involutive.
\end{enumerate}
A \textbf{complex $\mathbf{b}$-manifold} is a manifold with boundary equipped with a complex $b$-structure. An \textbf{isomorphism}  $\theta : M_1 \to M_2$ of complex $b$-manifolds is a diffeomorphism of the underlying manifolds with boundary satisfying ${^b\theta_*}({\bT^{0,1}}M_1) ={\bT^{0,1}}M_2$. Here, by abuse of notation, ${^b\theta_*}$ also denotes the complexification of the isomorphism ${^b\theta_*}:{\bT M_1} \to {\bT M_2}$ of \eqref{bpush}.
\end{defn}

Because ${\bT M}$ is naturally isomorphic to $TM$ over the interior, a $b$-complex structure on $M$ in particular gives a complex structure in the usual sense on the interior. In particular, one may consider ordinary complex manifolds as special cases of complex $b$-manifolds with empty boundary.   Actually, for any complex $b$-manifold $M$, the restricted complex structure on the interior   determines the whole $b$-complex structure, as the following proposition shows. Of course, not every complex structure on the interior  of $M$ is the restriction of a (unique) $b$-complex structure on $M$, just those that degenerate in a particular way at the boundary.

\begin{propn}\label{restriso}
Suppose $M_1,M_2$ are complex $b$-manifolds and $\theta : M_1 \to M_2$ is a diffeomorphism that  restricts to an isomorphism of (ordinary) complex manifolds on the interiors. Then, $\theta$ is an isomorphism of complex $b$-manifolds.
\end{propn}
\begin{proof}
We have an isomorphism of  complexified $b$-tangent bundles ${^b\theta_*} : \C {\bT} M_1 \to \C {\bT M_2}$. Applying the natural isomorphisms of $b$-tangent bundle and ordinary tangent bundle on the interiors, the subbundles ${^b\theta_*}({\bT^{0,1}}M_1)$ and ${\bT^{0,1}}M_2$ of $\C {\bT M_2}$ agree away from the boundary. The conclusion follows by a continuity argument.
\end{proof}

Taking duals, the splitting $\C{\bT M} = {\bT^{1,0}}M\oplus{\bT^{0,1}}M$ given by a complex $b$-structure induces  an associated splitting of the complexified $b$-cotangent bundle $\C {\bT M^*}$  and of the complexified $b$-exterior derivative $\bd : C^\infty(M,\C) \to C^\infty(M;\C {\bT M^*})$, as tabulated below:
\begin{align*}
&\C{\bT M^*}  ={\bT^{1,0}}M^* \oplus {\bT^{0,1}}M^* \\
& \bd = \bdel + \bdelbar \\ 
 &\bdel : C^\infty(M,\C) \to C^\infty(M;{\bT^{1,0}}M^*) \\ 
& \bdelbar  : C^\infty(M,\C) \to C^\infty(M;{\bT^{0,1}}M^*)
 \end{align*}
In effect, $\bdel f$, respectively $\bdelbar f$, is $df$ restricted to the $b$-$(1,0)$-vector fields, respectively the $b$-$(0,1)$-vector fields (that is to say, sections of ${\bT^{1,0}}M$, respectively sections of ${\bT^{0,1}}M$).

\begin{defn}
A \textbf{$\mathbf{b}$-holomorphic} function on a complex $b$-manifold $M$ (or an open subset thereof) is a smooth function $f:M\to\C$  satisfying $\bdelbar f =0$. We write $\bO(M)$ for the collection of $b$-holomorphic functions  on $M$.
\end{defn}
In more concrete terms, $f$ is $b$-holomorphic if $Xf = 0$ for every $b$-$(0,1)$-vector field $X$ or (what is sufficient), all $X$ in a given frame for ${\bT^{0,1}}M$. Because ${\bT^{0,1}}M$ is involutive, $\bO(M)$  is a  ring with respect to pointwise-multiplication.

\begin{ex}\label{basicex}
Let $M=\R^2_+=\C_+$, the closed $x \geq 0$ half space. Then, $\bX(M)=\langle x\partial_x,\partial_y \rangle$, the free $C^\infty(M)$-module generated by $x\partial_x$ and $\partial_y$. An example of a  complex $b$-structure for $M$ has ${\bT^{0,1}}M$  spanned by
\begin{align*}
\bdel_{\overline z}\coloneqq\tfrac{1}{2}(x\partial_x+i\partial_y)
\end{align*}
and, correspondingly, ${\bT^{1,0}}M$ spanned by
\begin{align*}
\bdel_z\coloneqq\tfrac{1}{2}(x\partial_x-i\partial_y).
\end{align*}
A function $f$ is $b$-holomorphic precisely when $\bdel_{\overline z}\, f=0$. An important globally-defined $b$-holomorphic function on $M$ was mentioned in the introduction:
\begin{align*}
g :M \to \C && g(x,y) = xe^{iy}.
\end{align*}
More generally, if $h$ is a usual holomorphic function defined near $0 \in \C$, then $f = h \circ g$ is a $b$-holomorphic function defined near $0 \in M$. There also exist $b$-holomorphic functions defined near $0$  not of the latter form and, indeed, not real-analytic near $0$. For example, on the domain $\{ (x,y)  \in M :   -\frac{\pi}{4}<y<\frac{\pi}{4}\}$, the function
\[ f(x,y) = \begin{cases}
\exp(\frac{-1}{g(x,y)}) & x>0 \\
0 & x=0
\end{cases} \]
is $b$-holomorphic. These examples are also noted in \cite{Barron-Francis}, Example~8.1.
\end{ex}

\begin{rmk}
Observe that the $b$-holomorphic function  $g(x,y)=xe^{iy}$ in the above example may be conceptualized as $g(x,y)=\phi_{iy}(x)$, where $\phi_t(x)=e^t x$ is the flow of $x \partial_x$. This somewhat cryptic remark will be expanded on in Section~\ref{generalpush}.
\end{rmk}

\begin{ex}\label{delnot}
Equip $\R^{2n+2}$ with coordinates  $(x_0,y_0,\ldots,x_n,y_n)=(z_0,\ldots,z_n)$ and let  $M=\R^{2n+2}_+$, the closed $x_0 \geq 0$ half space. Thus,
\begin{align*}
\bX(M)=\langle x_0\partial_{x_0},\partial_{y_0}, \partial_{x_1}, \partial_{y_1}, \ldots,\partial_{x_n}, \partial_{y_n}\rangle.
\end{align*}
An example of a  complex $b$-structure ${\bT^{0,1}}M$  for $M$ is the one spanned by:
\begin{align*}
\bdel_{\overline z_0} &\coloneqq\tfrac{1}{2}(x_0\partial_{x_0}+i\partial_{y_0}) \\
\partial_{\overline z_j}&\coloneqq\tfrac{1}{2}(\partial_{x_j}+i\partial_{y_j}) && j=1,\ldots,n. 
\end{align*}
Correspondingly, a frame for ${\bT^{1,0}}M$ is:
\begin{align*}
\bdel_{z_0}&\coloneqq\tfrac{1}{2}(x_0\partial_{x_0}-i\partial_{y_0}) \\
\partial_{z_j}&\coloneqq\tfrac{1}{2}(\partial_{x_j}-i\partial_{y_j}) && j=1,\ldots,n.
\end{align*}
The function $g(x_0,y_0,\ldots,x_n,y_n)=x_0e^{iy_0}$ from Example~\ref{basicex}   is still $b$-holomorphic, as are the complex  coordinate functions $z_j =x_j+iy_j$, $j=1,\ldots,n$. Additional $b$-holomorphic functions may be obtained by applying an $(n+1)$-variable holomorphic function (in the usual sense) to $g,z_1,\ldots,z_n$.
\end{ex}

\begin{ex}\label{tanex}
Let $N$ be the $x \geq 0$ half space $\R^2_+ = \C_+$ equipped with the complex $b$-structure for which ${\bT^{0,1}}N$ is spanned by:
\[ L \coloneqq \tfrac{1}{2} \Big( x \partial_x + i ( -xy \partial_x + (1+y^2) \partial_y) \Big). \]
This  complex $b$-manifold $N$ is isomorphic to a neighbourhood of the origin of the complex $b$-manifold $M=\R^2_+$ from Example~\ref{basicex} for which  ${\bT^{0,1}}M$ was spanned by $\bdel_{\overline z} \coloneqq \frac{1}{2}(x\partial_x + i \partial_y)$. Indeed, the diffeomorphism $G : \R_+ \times (-\frac{\pi}{2},\frac{\pi}{2})\to\R^2_+$ defined by $G(x,y)=(x\cos y, \tan y)$ satisfies $G_*(\bdel_{\overline z}) = L$. Thus, $b$-holomorphic functions on $N$ may be obtained by pushforward through $G$. For example, the pushforward of $g(x,y)=xe^{iy}$ by $G$ is $g'(x,y)=x+ixy$ which indeed satisfies $Lg'=0$. 
\end{ex}

For a general complex $b$-manifold $M$, it is not immediately clear whether, locally near points on the boundary, any nonconstant $b$-holomorphic functions must exist at all. This article will confirm that they do (Sections~\ref{sec:2exist} and \ref{sec:exist}).

Note that a $b$-holomorphic function $f$ on a complex $b$-manifold $M$ restricts to a  holomorphic function in the usual sense on the interior. Indeed, in a similar spirit to Proposition~\ref{restriso}, if $f$ is smooth on $M$ and restricts to a  holomorphic function on the interior, then it is $b$-holomorphic on $M$ by a continuity argument.

Recall that, in the case of ordinary complex manifolds (defined using integrable almost-complex structures), the existence of local holomorphic functions is closely-tied  to the existence of charts. The proposition below, which will be used in Sections~\ref{sec:2main} and \ref{sec:main}, illustrates this principle. We  include its (standard) proof for the sake of completeness.

\begin{propn}\label{holocharts}
Let $M$ be a complex manifold. Let $\Omega \subseteq M$ and $\Omega'\subseteq \C^n$ be open. Suppose $f_1,\ldots,f_n:\Omega \to \C$ are holomorphic functions such that  $\theta(p) = (f_1(p),\ldots,f_n(p))$ defines a diffeomorphism $\theta : \Omega \to \Omega'$. Then,  $\theta_*$ maps $T^{0,1}M$ onto $T^{0,1}\C^n$ over $\Omega$ where, by definition, the latter is  spanned by $\partial_{\overline z_j} \coloneqq \frac{1}{2}(\partial_{x_j} + i \partial_{y_j})$, $j=1,\ldots,n$.
\end{propn}
\begin{proof}
Let $L \in C^\infty(\Omega;T^{0,1}M)$ be  a $(0,1)$-vector field defined on $\Omega$.  Using $L f_j=0$ for $j=1,\ldots,n$ and $\theta_*(f_j)=z_j$ (the $j$th coordinate function) for $j=1,\ldots,n$, we obtain 
\[ \theta_*(L)z_j=0. \] 
If we write $\theta_*(L) = \sum_{j=1}^n \alpha_j \partial_{z_j} + \beta_j \partial_{\overline z_j}$, where $\alpha_j,\beta_j$ are smooth complex-valued functions on $\Omega'$, then the above  gives $\alpha_j=0$ for $j=1,\ldots,n$. Thus, we obtain
\[ \theta_*(L) =  \sum_{j=1}  \beta_j \partial_{\overline z_j}, \]
so that $\theta_*(L)$ is a $(0,1)$-vector field on $\Omega'$. 
\end{proof}

Our arguments follow a similar paradigm in that we first construct a supply of $b$-holomorphic functions and then use those functions to build coordinate charts. Note, however, that in the case of complex $b$-manifolds, $b$-holomorphic functions cannot directly play the role of coordinate functions of charts. To see this, note that any solution $f : \R^2 \to \C$ of $\bdel_{\overline z} f=0$ (where $\bdel_{\overline z}\coloneqq \frac{1}{2}(x\partial_x+i\partial y)$) is necessarily constant along $\{0\}\times \R$.

\section{Existence of \texorpdfstring{$b$}{b}-holomorphic functions: dimension \texorpdfstring{$2$}{2}}\label{sec:2exist}

In this section, we show that, near any point on the boundary of a complex $b$-manifold of real dimension $2$, there is a nontrivial $b$-holomorphic function. Here, ``nontrivial'' means the derivative normal to the boundary is nonzero. From Theorem~\ref{Mendoza}, this amounts to the following (as usual $\bdel_z \coloneqq \frac{1}{2}(x\partial_x-i\partial_y)$, $\bdel_{\overline z} \coloneqq \frac{1}{2}(x\partial_x+i\partial_y)$).

\begin{thm}\label{2dexist}
Let $\R^2_+$ be the closed $x \geq 0$ half plane in $\R^2$.
Let $M \subseteq \R^2_+$ be a relative neighbourhood of the origin   equipped with a complex $b$-structure  ${\bT^{0,1}}M$ spanned by:
\begin{align*}
L \coloneqq \bdel_{\overline z} + \gamma \cdot \bdel_z
\end{align*} where $\gamma : M \to \C$ is a smooth function vanishing to infinite order on $\{0\}\times \R$. Then, there is a neighbourhood $U \subseteq M$ of the origin and a $b$-holomorphic function $f:U \to \C$ that vanishes  on $\{0\}\times \R$ and satisfies $\partial_x f(0,0)=1$. 
\end{thm}

The proof of Theorem~\ref{2dexist} will rely on the following   three  elementary lemmas whose proofs we omit. 

Firstly, we record several vector field pushforward formulae for the singular coordinate change $g(x,y)=xe^{iy}$. Note $g$  is simply  the polar coordinate transformation; it is a local diffeomorphism away from $\{0\}\times\R$.  We  treat a more general type of coordinate change in Section~\ref{generalpush}, providing additional context for the methods of this section and rendering them somewhat less ad hoc.

\begin{lemma}\label{pushforwards}
The smooth surjection $g : \R^2 \to \C$ defined by $g(x,y)=xe^{iy}$ satisfies:
\begin{align*}
g_*(x \partial_x) &= x \partial_x + y \partial_y  &
g_*(\bdel_{z}) &=  z \partial_{ z}  \\
g_*(\partial_y) &= -y \partial_x + x \partial_y & 
g_*(\bdel_{\overline z}) &= \overline z \partial_{\overline z}.
\end{align*}
Here, $\bdel_{\overline z} \coloneqq \frac{1}{2}(x\partial_x+i\partial_y)$, $\bdel_{z} \coloneqq \frac{1}{2}(x\partial_x-i \partial_y)$, $\partial_{\overline z} \coloneqq \frac{1}{2}(\partial_x + i \partial_y)$, $\partial_z \coloneqq \frac{1}{2}(\partial_x - i \partial_y)$.  
\qed
\end{lemma}

Secondly, we need the following  fact concerning the expression in polar coordinates of plane functions vanishing to infinite order at the origin.

\begin{lemma}\label{pushlem}
Suppose $\gamma : \R^2_+  \to \C$ is a smooth function vanishing to infinite order on $\{0\}\times\R$ that is $2\pi$-periodic in $y$. Let $g(x,y)=xe^{iy}$, as above. Then, the pushforward $g_*(\gamma)$ is a well-defined smooth function on $\C$ vanishing to infinite order at $0$. \qed 
\end{lemma}

Thirdly, we need the following  divisibility property of plane functions vanishing to infinite order at the origin. 

\begin{lemma}\label{divzbar}

Let $\gamma$ be a smooth complex-valued function on $\C$ vanishing to infinite order at $0$.  Then, $(1/z) \gamma$ and $(1/\overline z)\gamma$  extend to smooth complex-valued functions on $\C$ vanishing to infinite order at $0$.\qed
\end{lemma}

We now  prove the main result of this section.

\begin{proof}[Proof of Theorem~\ref{2dexist}]
Since we are only looking for a local solution, there is no harm in assuming that $M = \R^2_+$ and that $\gamma$ is $2\pi$-periodic as in Lemma~\ref{pushlem}. Indeed, we may take $\gamma$ to be compactly-supported in $\R_+ \times (-\frac{\pi}{2},\frac{\pi}{2})$  and extend it periodically. Then, the pushforward of $L$ by $g(x,y)=xe^{iy}$ is well-defined and smooth, as the following computation shows.
\begin{align*} 
g_*(L) 
&= g_*(\bdel_{\overline z}+\gamma \cdot \bdel_z) \\
&= \overline z \partial_{\overline z}  + g_*(\gamma) z \partial_z && \text{(Lemmas~\ref{pushforwards} and \ref{pushlem}).} 
\end{align*}

Because $g$ is a local diffeomorphism away from $\{0\}\times\R$, we have that $g_*(L)$ defines a complex structure on $\C \setminus \{0\}$. Furthermore, using  Lemma~\ref{divzbar}, we can write $g_*(\gamma) = \overline z \gamma'$ where $\gamma'$ is another smooth complex-valued function on $\C$ vanishing to infinite order at $0$. Thus, 
\[g_*(L)  =  \overline z ( \partial_{\overline z}  + \gamma'  z\partial_z). \]
But now, note that $\partial_{\overline z}  + \gamma'z \partial_z$ defines an (ordinary) complex structure on all of $\C$  (because the $\gamma'z \partial_z$ term vanishes at $0$). So, by the ordinary Newlander-Nirenberg theorem for dimension two, there exists a smooth complex-valued function $h$ defined near $0 \in \C$ such that 
\[ (\partial_{\overline z}  + \gamma' z \partial_z)h=0. \]
We may furthermore choose $h$  to satisfy  $h(0)=0$ and $\partial_x h(0)=1$.  Putting $f=h \circ g$ completes the proof.
\end{proof}

\begin{rmk}\label{techniquesneeded}
At the outset of the  proof above, we implicitly used the fact that  integrability is automatic  for complex structures/complex $b$-structures  of real dimension 2. In higher dimensions this is not the case and we cannot use bump functions to modify structures outside of compact sets so casually. Section~\ref{new} is devoted to developing a workaround for this precisely this issue.  
\end{rmk}

\section{Proof of main result for dimension \texorpdfstring{$2$}{2}}\label{sec:2main}

This section is devoted to the proof of  Theorem~\ref{mainthm} in the two-dimensional case ($n=0$, in the notation of Theorem~\ref{mainthm}). Thanks to Mendoza's Theorem~\ref{Mendoza}, we may take our complex $b$-manifold    to be a relative neighbourhood of the origin $M\subseteq \R^2_+$ 
with ${\bT^{0,1}}M$ spanned by:
\begin{align*}
L \coloneqq \bdel_{\overline z} + \gamma \cdot \bdel_z
\end{align*} where $\gamma : M \to \C$ is a smooth function vanishing to infinite order on $\{0\}\times\R$.  Our task is to find new coordinates near $(0,0)$ in which ${\bT^{0,1}}M$ is spanned by $\bdel_{\overline z}$.  From Theorem~\ref{2dexist}, there is an  open set $U \subseteq M$ containing $(0,0)$ and a smooth function $f : U \to \C$ such that:
\begin{enumerate}[(i)]
\item $L f =0$ ($b$-holomorphicity),
\item $f$ vanishes on $\{0\}\times \R$,
\item $f_x(0,0)=1$ (denoting partial derivatives by  subscripts for brevity).
\end{enumerate}
We split $f$ into its real and imaginary parts
 \[ f = u+iv =  (u,v) \]
(recall $\C=\R^2$ in this article). The key  will be the local  coordinate change $F$ defined, roughly speaking, by $(u,\frac{v}{u})$. In more precise terms, from (ii),  we may write
\begin{align*}
u(x,y) = x\cdot  a(x,y) && v(x,y) =  x \cdot b(x,y) 
\end{align*}
where $a$ and $b$ are smooth real-valued functions on $U$. From (iii), we have
\begin{align*}
a(0,0)=u_x(0,0) = 1 && b(0,0)=v_x(0,0)=0.
\end{align*}
Shrinking $U$, we may assume $a$ is nowhere-vanishing on $U$ and thus define 
\begin{align*}
F:U \to \R^2 && F(x,y)= \left( u(x,y) , \frac{b(x,y)}{a(x,y)} \right).
\end{align*}
\begin{claim}\label{diffeo}
After possibly further shrinking $U$ around $(0,0)$, one has:
\begin{enumerate}[(a)]
\item $F$ is a diffeomorphism from $U$ onto a relative open set $F(U) \subseteq \R^2_+$, 
\item $F(0,0)=(0,0)$ and $F(U\cap (\{0\} \times \R)) = F(U)\cap (\{0\} \times \R)$,
\item $f =  \kappa \circ F$,  where $
\kappa : \R^2 \to \R^2$ is given by  $\kappa(x,y)=(x,xy)$.
\end{enumerate}
\end{claim}
\begin{proof}
Statement (c) is a direct consequence of the definition of $F$. For statement (b), since $u$ vanishes on $\{0\}\times\R$ and $u_x(0,0)=1$, we may, after shrinking $U$, assume $u^{-1}(0) \subseteq \{0\}\times \R$.  For statement (a), it suffices to show the Jacobian of $F$ at  $(0,0)$ is invertible. Then, by the inverse function theorem, after possibly shrinking $U$ again,  $F$ gives a diffeomorphism $U \to F(U)$.  We claim that the Jacobian of $F$ at $(0,0)$ has the form $\begin{bmatrix} 1&0 \\ * & 1 \end{bmatrix}$. Obviously the top row is correct, so we need only confirm that $\partial_y\Big(\frac{b}{a}\Big)(0,0)=1$. Away from $\{0\}\times\R$, 
\begin{align*}
\partial_y\Big(\frac{b}{a}\Big) = \partial_y\Big(\frac{v}{u}\Big) = \frac{v_yu-vu_y}{u^2}.
\end{align*}
Collecting real and imaginary parts in $Lf=0$ we have
\begin{align*}
xu_x - v_y \sim 0 && xv_x+u_y \sim 0, 
\end{align*}
where $\sim$ denotes agreement of Taylor series at $(0,0)$. Thus, 
\begin{align*}
v_yu-vu_y  \sim (xu_x)u + v (x v_x) = x^2 (u_x a + v_x b )  
\end{align*}
from which it follows that
\begin{align*}
\partial_y\Big(\frac{b}{a}\Big) \sim \frac{u_xa+v_xb}{a^2}.
\end{align*}
In particular, $\partial_y\Big(\frac{b}{a}\Big)(0,0) = \frac{(1)(1)+(0)(0)}{(1)^2} = 1$,  as was claimed.
\end{proof}

It will be relevant to see what the above constructions amount to  in the model case when  $\gamma=0$ and $L$ is simply  $\bdel_{\overline z}= \frac{1}{2}(x \partial_x + i \partial_y)$. Then, in place of $f$, we may use our function 
\[ g(x,y)=xe^{iy} = (x \cos y , x \sin y). \] 
With this replacement, the diffeomorphism $F$ above becomes the diffeomorphism $G$ appearing in Example~\ref{tanex}
\begin{align*}
G : [0,\infty) \times (-\tfrac{\pi}{2},\tfrac{\pi}{2}) \to \R^2_+ &&   G(x,y) = (x \cos y, \tan y).
\end{align*}
Define $W\coloneqq F(U)$ and $V \coloneqq G^{-1}(W)$, so that we have  diffeomorphisms
\[
\begin{tikzcd}
U \ar[r,"{F}"] & 
W 
&V. \ar[l,"{G}"']
\end{tikzcd}
\]
Applying the induced maps ${^bF_*}, {^bG_*}$ on $b$-tangent bundles (see \eqref{bpush} in Section~\ref{sec:background}), we arrive at two complex $b$-structures on $W$ which can be compared.

\begin{claim}
The complex $b$-structures on $W$ spanned by ${^bF_*}(L)$ and ${^bG_*}(\bdel_{\overline z})$ are equal.
\end{claim}

\begin{proof}
Referring to  Proposition~\ref{restriso}, it suffices to check this away from the boundary. Thus, writing $Z \coloneqq \{0\} \times \R$, we  need only check that the usual pushforwards $F_*(L)$ and $G_*(\bdel_{\overline z})$ coincide up to a smooth rescaling over the interior $W \setminus Z$.   Note  $\kappa(x,y)=(x,xy)$ restricts to a diffeomorphism of $\R^2 \setminus Z$ and recall $f = \kappa \circ F$ on $U$,  $g = \kappa \circ G$ on $V$. Thus, we have  diffeomorphisms 
\[
\begin{tikzcd}
U \setminus Z \ar[r,"{f}"] & 
\kappa(W \setminus Z) 
& V \setminus Z, \ar[l,"{g}"']
\end{tikzcd}
\]
and it suffices to check that  $(f|_{U\setminus Z})_*(L)$ and $(g|_{V\setminus Z})_*(\bdel_{\overline z})$ agree up to a smooth rescaling. Indeed, since $f$ and $g$ are holomorphic in the usual sense away from $Z$, Proposition~\ref{holocharts} shows that $(f|_{U\setminus Z})_*(L)$ and $(g|_{V\setminus Z})_*(\bdel_{\overline z})$ are rescalings of $\partial_{\overline z} = \frac{1}{2}(\partial_x + i\partial_y)$ (actually, from Lemma~\ref{pushforwards}, we even know $g_*(\bdel_{\overline z}) = \overline z \partial_{\overline z}$).
\end{proof}

 The preceding claim completes this section's proof of  Theorem~\ref{mainthm} in the two-dimensional case. The desired local coordinates are provided by $G^{-1} \circ F : U \to V$.

\section{Extending deformed structures across boundaries}\label{new}

As noted in Remark~\ref{techniquesneeded}, additional methods are needed to extend our results to higher-dimensional cases. In this section, we prepare the ground by proving a technical result  (Proposition~\ref{sector}) to the effect that certain small deformations of the standard complex structure defined near the origin in a sector $\Omega_\theta = \{(z_0,\ldots,z_n) \in \C^{n+1}: |\Arg(z_0)| < \theta\}$ can be extended integrably into a neighbourhood of the origin. 
Our proof relies on Catlin's Newlander-Nirenberg theorem for complex manifolds with pseudoconvex boundary \cite{catlin}. We first digress  to clarify  what exactly constitutes a complex manifold with  boundary.  
As observed by Hill, there are two generally distinct notions:

\begin{defn}[]
A \textbf{complex manifold  with abstract boundary} is a real manifold with (smooth) boundary $M$ equipped with an involutive  almost complex structure\footnote{Just as in the case where boundary is absent, an ``almost complex structure'' is a subbundle $T^{0,1} M$ of the complexified tangent bundle  satisfying $\C TM = T^{1,0}M\oplus T^{0,1}M$ where $T^{1,0}M \coloneqq \overline{T^{0,1} M}$, and ``involutive'' means the  module of smooth sections is invariant under Lie bracket.}. We say   $p \in \partial M$ is a  \textbf{concrete boundary point}  if there  is a diffeomorphism from a neighbourhood   of $p$ onto a set of the form  $\{ x \in \Omega : \rho(x) \leq 0 \}$ where $\Omega \subseteq \C^d$ is open and $\rho : \Omega \to \R$ is a smooth map having $0$ as a regular value such that, under this chart identification,  $T^{0,1} M$ is framed by the standard antiholomorphic vector fields on $\C^d$. 
\end{defn}

\begin{rmk}
It deserves to be emphasized  that, unlike in the real setting,  boundaries of complex  manifolds  are typically not homogeneous. Curvature invariants along  the boundary can   obstruct biholomorphic identification of neighbourhoods of different boundary points.  This strongly contrasts the main finding of this article, that a complex $b$-manifold has no local invariants along its boundary. Thus, although the settings of complex manifolds with boundary and complex $b$-manifolds may appear superficially similar in some ways, they  are completely different in this respect.
\end{rmk}

In \cite{Hill}, Hill gave examples showing that abstract boundary is  a  more general concept than concrete boundary. Near an abstract boundary point, it may occur that there are not sufficiently many holomorphic functions smooth up to the boundary to facilitate the existence of  charts of the required type. In other words, the most optimistic analog of the Newlander-Nirenberg theorem is false for complex manifolds with boundary. An important extra hypothesis under which an analogue of the Newlander-Nirenberg theorem does hold is pseudoconvexity.  The    machinery of the Dolbeault   complex still works perfectly well on a complex manifold with abstract boundary and,  accordingly, one may define pseudoconvexity  as usual  in terms of  positivity of the Levi form.

\begin{defn}
Let $M$ be a complex manifold with abstract boundary. Fix $p \in \partial M$ and let $\rho : W \to \R$ be a smooth function on a neighbourhood $W \subseteq M$ of $p$ that is negative on the interior, vanishes on the boundary and has $0$ as a regular value. We say that $M$ has \textbf{pseudoconvex} boundary  at $p$  if $\overline \partial \partial \rho (\overline v,v) = \partial \overline \partial \rho (v,\overline v) \geq 0$ for all $v \in T^{1,0}_p \partial M \coloneqq T^{1,0}_p M \cap \C T_p \partial M$. This notion is independent of the choice of defining function.
\end{defn}

It was shown by Catlin \cite{catlin} that, if the boundary is pseudoconvex at $p$, then $p$ is a concrete boundary point.   Catlin's result can be cast as the following extension principle.

\begin{thm}[\cite{catlin}]\label{catlinthm}
Let $M$ be a $2d$-dimensional real manifold with boundary smoothly embedded in $\C^d$. Assume $M$ is equipped with an involutive almost complex structure (not the standard one inherited from $\C^d$) and that $\partial M$ is pseudoconvex with respect to this structure near  a point $p \in \partial M$. Then, there exists an open  neighbourhood $U \subseteq \C^d$ of $p$ and a complex structure on $U$ that extends the given structure on $U \cap M$. 
\end{thm}

Note it is elementary that such an extension   of the underlying almost complex structure on $M$ exists, more or less because of the way smooth functions on  manifolds with boundary are defined. The crux is of course to ensure that  the  extended structure is  involutive.

\begin{proof}
By the main theorem of \cite{catlin}, there  exists a neighbourhood $U \subseteq \C^d$ of $p$ and smooth functions $f_1,\ldots,f_d$ on $U \cap M$ whose differentials are linearly independent at $p$ that are holomorphic for the given structure. Smoothly extend these  functions arbitrarily into $U$. Shrinking $U$ around $p$, we may assume these functions define a diffeomorphism of $U$ onto an open subset of $\C^d$  which is moreover an isomorphism of almost complex structures on  $U \cap M$. Pulling back   the complex structure of $\C^d$ through this diffeomorphism yields a complex structure on $U$ extending the given structure on $U \cap M$. 
\end{proof}

We will want to apply the above theorem in the following situation:
\begin{enumerate}
\item  $M$ is embedded in  the $x_0 \geq 0$ half space $\C^{n+1}_+$ as a convex set with smooth boundary.
\item The boundary $\partial M$ contains the origin $p=0$ as well as  a neighbourhood of nearby points in the complex hyperplane $Z_c  = \{0\} \times \C^n$. 
\item  $M$ is equipped with an involutive almost complex structure that is a  deformation by terms that vanish to infinite order on $Z_c  =\{0\} \times \C^n$ of the standard structure inherited from $\C^{n+1}$.
\end{enumerate}
If $M$ could    additionally be chosen to  be strictly convex, and hence strictly pseudoconvex with respect to the standard complex structure, then  by continuity  it would also be strictly pseudoconvex for the  deformed structure sufficiently close to $Z_c  = \{0\}\times\C^n$. However, because strict convexity/pseudoconvexity   of $\partial M$ certainly fails  where it overlaps  the hyperplane $Z_c $ and nonstrict pseudoconvexity is not an open condition, things are not quite so simple. Nonetheless, we can finesse our way around this apparent issue by  instead choosing $M$ such that:  
\begin{enumerate}[resume]
\item $M$ is strictly convex away from $Z_c  =\{0\} \times \C^n$ and, moreover, the least principal curvature  of the boundary does not vanish too quickly on $Z_c $ relative to the vanishing rate of the deformation terms of the given complex structure.
\end{enumerate}
For a domain $M$ satisfying (1), (2), (3), (4) it still follows by continuity that $\partial M \setminus Z_c $ is strictly pseudoconvex for the deformed structure near $Z_c $ so that  $\partial M$ is pseudoconvex on and near  $Z_c $, allowing   Theorem~\ref{catlinthm} to be applied. Since formulating a general statement along these lines seems unlikely to provide much added clarity and only one example of such a domain is actually necessary for our purposes, it seems more appropriate to confine our discussion to an explicit example. The domain  $K$ selected for use below makes calculations relatively straightforward, but is otherwise entirely ad hoc. 

Recall we have been coordinatizing $\C^{n+1} = \R^{2n+2}$ as $(z_0,\ldots,z_n)=(x_0,y_0,\ldots,x_n,y_n)$. Here we will drop the subscript on the leading coordinates $z_0=x_0+iy_0$ and abbreviate the remaining coordinates by $w$  to reduce clutter:
\begin{align*}
z =  z_0&&x = x_0 &&  y = y_0 && w = (z_1,\ldots,z_n).
\end{align*}
We also   introduce an auxiliary function $r$ defined by:
\[ r \coloneqq 1 - |w|^2=1-(|z_1|^2+\ldots+|z_n|^2). \]

\begin{lemma}\label{pseudobound}
The domain $K  \coloneqq \{ (z,w) \in \C^{n+1}  : |w| < 1, |z -r | \leq r \}$ is pseudoconvex with respect to the standard complex structure and,  moreover, there exists   $\alpha > 0$ such that  
\begin{align*} \partial \overline \partial \rho (v,\overline v)  \geq \alpha x |v|^2 
\end{align*}
for all $p=(z,w)=(x,y,w) \in \partial K$ and all $v \in T^{1,0}_p \partial K$.
\end{lemma}

\begin{figure}
    \centering
    \includegraphics[width=0.5\linewidth]{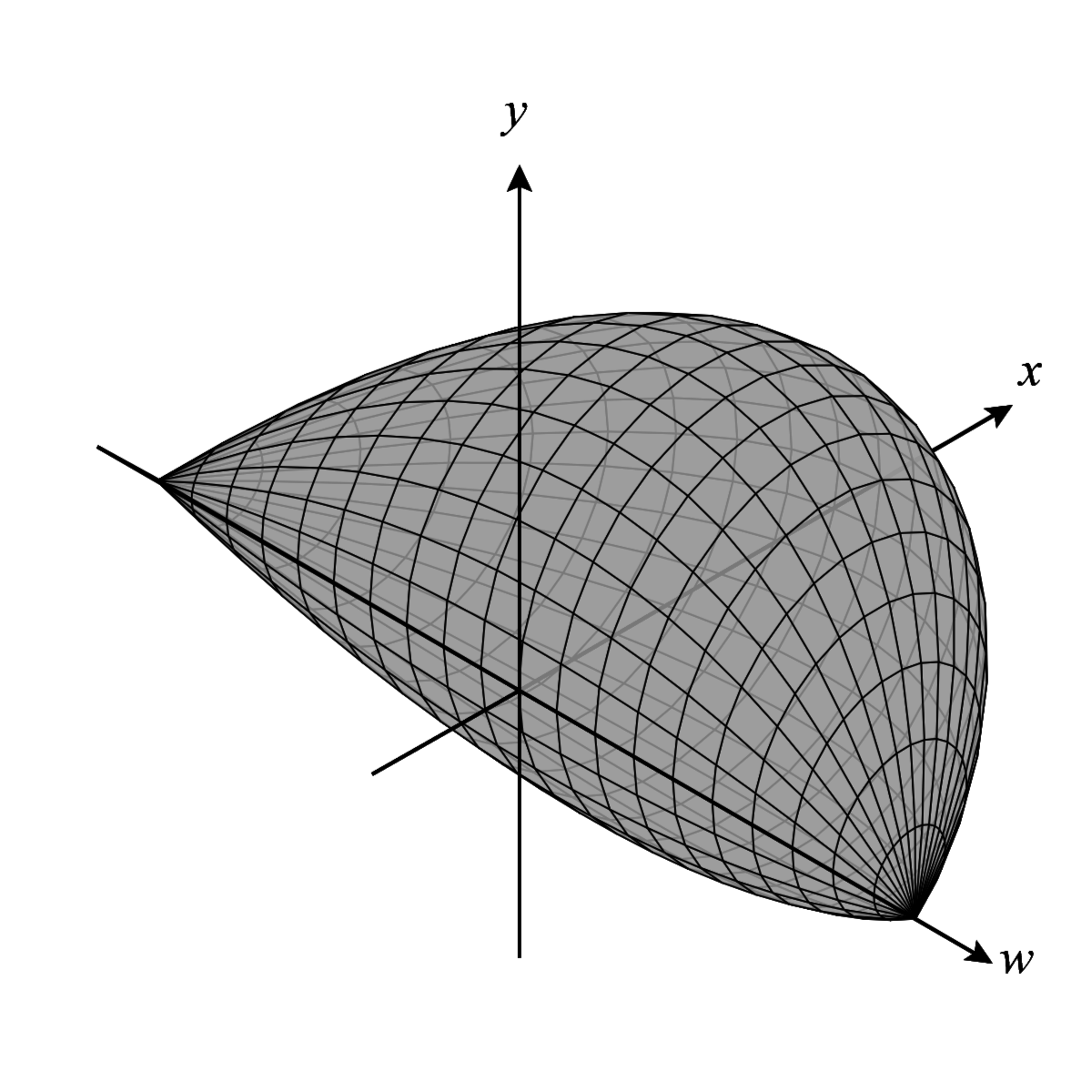}
    \caption{Visualization of the  domain $K$ in Lemma~\ref{pseudobound}. }
    \label{croissant}
\end{figure}

Above, $|v|$ denotes the norm associated  to the standard inner product on the complexified tangent bundle of $\C^{n+1}$ for which   $\partial_{z_j},\partial_{\overline z_j}$, $j=0,\ldots,n$  constitute an orthonormal frame.

\begin{proof}
Observe $K$ admits an obvious defining function:
\begin{align*}
\rho : \Omega \to \R &&  \rho(z ,w) = |z -r|^2 -r^2 = |z |^2 - (z +\overline z )r=|z |^2-2x r 
\end{align*}
where $\Omega \subseteq \C^{n+1}$ is the open set:
\[ \Omega \coloneqq \{ (z,w) \in \C^{n+1} : |w|< 1\}.\] 
Take  note that  $r=1-|w|^2$ is positive on $\Omega$ and $|z-r|=r$ on $\partial K$. 

We calculate:
\begin{align*} 
\partial \rho &= (\overline z  - r) \, dz  + 2x  \sum_{j=1}^n \overline z_j \, dz_j \\
\partial \overline \partial \rho &= dz  \, d \overline z  + \sum_{j=1}^n \Big(z_j \,  dz  \, d \overline z_j  + \overline z_j \,  d z_j \,  d \overline z  + 2x  \, d z_j \, d \overline z_j \Big) .
\end{align*}
If $p=(z ,w) \in \partial K$, then $|z -r| = r \neq 0$  so any vector  $v \in T^{1,0}_p K$ can be expressed as:
\begin{align} \label{v} v &= \lambda (z -r) \partial_{z } +  \sum_{j =1}^n r u_j \partial_{z_j} && \text{ where $u = (u_1,\ldots,u_n) \in \C^n$ and $\lambda \in \C$.}
\end{align}
This expression    \eqref{v} for $v$  leads to:
\begin{align}
|v|^2 &= r^2 (|\lambda|^2 + |u|^2) \label{|v|} \\
\partial \rho (v) &= \lambda |z -r|^2 +  2rx  \sum_{j=1}^n \overline z_j u_j  \label{drho} \\
&= \lambda r^2 +  2rx  \< w , u\>  \nonumber \\
\partial \overline \partial \rho (v,\overline v) &= |\lambda|^2 |z -r|^2 + \sum_{j=1}^n \Big( \lambda (z -r) r \overline u_j z_j + \overline \lambda (\overline z  -r ) r \overline z_j u_j + 2 r^2 x  \overline u_j u_j \Big)  \label{ddrho}  \\
&= |\lambda|^2r^2 + 2\Re\Big( \lambda r (z -r) \<u,w\> \Big) + 2r^2x  |u|^2. \nonumber \nonumber 
\end{align}
Above we are using the convention that the inner product $\<\cdot,\cdot\>$ on  $\C^n$ is conjugate linear in the first slot.
From \eqref{drho}, we have  that $v \in T^{1,0}_p K$ of the form \eqref{v} belongs to $T^{1,0}_p\partial K$  if and only if the following tangency condition holds:
\begin{align} \label{tangency} 
\lambda   r =- 2x  \<w,u\>. 
\end{align}
For such $v \in T^{1,0}_p\partial K$, we have
\begin{align*} |v|^2 &= r^2|\lambda|^2 + r^2|u|^2 && \text{using \eqref{|v|}} \\
&= 4x ^2 |\< w, u \>|^2 +r^2 |u|^2  && \text{using \eqref{tangency}} \\
&\leq 4(2r)^2 |u|^2 + r^2 |u|^2 && |w| < 1 \text{ and } 0 \leq x \leq 2r \\ 
& =17r^2 |u|^2.
\end{align*}
and so
\begin{align*}
\partial \overline \partial \rho (v,\overline v) 
&= |\lambda|^2r^2 + 2\Re\Big( \lambda r (z -r) \<u,w\> \Big) + 2r^2x  |u|^2 && \text{using \eqref{ddrho}} \\
&= 4x ^2|\<w,u\>|^2  -2 \Re\Big( 2x |\<w,u\>|^2 (z  - r) \Big) + 2r^2 x  |u|^2 && \text{using \eqref{tangency}}\\
&= 4 rx |\<w,u\>|^2 + 2r^2 x  |u|^2 && \Re(z)=x \\
&\geq 2r^2x|u|^2 \\
&\geq \frac{2}{17} x |v|^2,
\end{align*}
as required.
\end{proof}

\begin{rmk} 
As Figure~\ref{croissant} suggests, $K$ is actually strictly convex away from $Z_c  =\{0\} \times \C^n$. Indeed, with a bit more work, one may obtain a similar  lower bound on the extrinsic curvature of $\partial K$ and deduce the  bound in Lemma~\ref{pseudobound} as a consequence. Since working directly with the Levi form leads to simpler calculations, we choose not to bury the lede in this way.
\end{rmk}

We now use the bound in Lemma~\ref{pseudobound} to argue that if the structure on $K$ is deformed by terms which vanish quickly on $Z_c $ then pseudoconvexity is not disrupted near $Z_c $, leading to the following:

\begin{propn}\label{specialdomain}
Equip the domain $K = \{ (z,w) \in \C^{n+1} : |w|<1, |z-r| \leq r\}$ of  Lemma~\ref{pseudobound} with an involutive almost  complex structure admitting a global frame of the form $\partial_{\overline z_j} + \Gamma_j$, $j=0,\ldots,n$ where $\Gamma_j$ is a   smooth complex vector field on $K$ vanishing to infinite order along the complex hyperplane $Z_c  = \{0\} \times \C^n$. Then, there exists a complex structure defined on a neighbourhood $U \subseteq \C^{n+1}$ of the origin extending the given structure on $U \cap K$. 
\end{propn}

\begin{proof}
Let us write $K$ and $K'$ to refer to the domain equipped with, in the first place, the standard complex structure inherited from $\C^{n+1}$ and, in the second place, the given deformation of that structure.  Let  $\rho : \Omega \to \R$ be the defining function for the domain used in the proof of  Lemma~\ref{pseudobound}. We write $\L$  for the Levi form $\partial \overline \partial \rho$ of  $\rho$  with respect to the standard complex structure and $\L'$ for the Levi form with respect to the  deformed   structure. By hypothesis:
 \begin{enumerate}[(i)]
 \item The $2$-forms $\L$ and $\L'$  agree to infinite order along $Z_c  = \{0\} \times \C^n$. 
 \item The bundles $T^{1,0}K$ and $T^{1,0}K'$ agree to infinite order on along $Z_c $.
 \end{enumerate}  
To make sense   of (ii), note both bundles are subbundles  of the domain's complexified tangent bundle. We assume the latter  to be equipped with its  standard inner product, i.e. the one for which $\partial_{z_j},  \partial_{\overline z_j}$, $j=0,\ldots,n$ are an orthonormal basis.

Let us recall an elementary argument from linear algebra: call two $k$-dimensional subspaces $E,F \subseteq \C^d$   ``close'' if their associated orthogonal projections $P,Q$ satisfy $\| P- Q\| <1$ with respect to the Euclidean operator norm.  Recall one may assign  a unitary operator $\Phi_{F,E}$ to each close pair in such a way that  $\Phi_{F,E} (E) = F$, $\Phi_{E,E}=\id$ and $\Phi_{F,E}$ depends smoothly on $E$ and $F$. Explicitly,  one may define $\Phi_{F,E}$ to be the unitary part of the polar decomposition of the invertible operator $\frac{1}{2}((2Q-\id)(2P-\id)+\id)$. 

Using (ii) and the elementary argument just recalled, we have the following:
\begin{enumerate}[(i),resume]
\item There exists a relatively open set $W \subseteq \partial K$ containing $Z_c $ over which is defined a smooth fiberwise unitary operator $v \mapsto v' : T^{1,0} \partial K \to T^{1,0} \partial K'$ which agrees with the identity operator to infinite order along $Z_c $. 
\end{enumerate}
Next, combining (i) and (iii), we can conclude the following:
\begin{enumerate}[(i),resume]
\item There is a function $\epsilon : \partial K \to [0,\infty)$ which vanishes to infinite order on $Z_c $ such that:
\begin{align*}
 \left|\L_p(v,\overline v) - \L'_p (\cramped{v'},\overline{v'} \, )\right| \leq \epsilon(p) |v|^2 && \text{ for all } p \in \partial K , v \in T^{1,0} \partial K.
\end{align*}
\end{enumerate}
Notice that the restriction of the coordinate function $x$ to $\partial K$ is positive away from $Z_c $ and, moreover, vanishes to exactly second-order on $Z_c $. Thus, from (iv) and the bound in Lemma~\ref{pseudobound}, it follows that, near $Z_c $, the deformed complex structure on $K'$ satisfies the same sort of convexity bound satisfied by $K$. That is, there exists a neighbourhood $U_1 \subseteq K'$ containing $K' \cap Z_c $ and a constant $\beta >0$ such that 
\[ \L'_p(v,\overline v) \geq \beta x |v|^2 \]
for all $p=(z,w)=(x,y,w) \in U_1$ and all $v \in T^{1,0}_p \partial K'$. In particular, $K'$ has pseudoconvex boundary in a neighbourhood of the origin. The desired conclusion is now a consequence of Theorem~\ref{catlinthm}.
\end{proof}

It remains to explain why Proposition~\ref{specialdomain}, ostensibly concerned with the special domain $K$, also implies our desired result that certain small  deformations of the standard complex structure defined  in a sector $\Omega_\theta =\{(z,w) :|\Arg(z)| < \theta\}$ can be extended (integrably) into a neighbourhood of the origin. The basic idea is to first widen $\Omega_\theta$ to an ``obtuse sector'' by applying a power map to the $z$-coordinate. By pushforward, we then have a complex structure on $K \subseteq \C^{n+1}_+$ which we may extend into a neighbourhood of the origin using Proposition~\ref{specialdomain} and then pull back by the same power map. The complex structure so obtained   may not match  the original one on the whole of  $\Omega_\theta$, but it will on a smaller subsector.

Rather than deal with complex structures defined in specific regions of $\C^{n+1}$ below, we find it more convenient to  instead deal with almost complex structures which  we  may take to be defined everywhere, but only integrable within certain regions. This allows us to avoid excessive mumbling about domains of definition and, in cases where those domains do not have smooth boundaries, the meaning of smoothness at the corners.

In the proposition below, several verifications are left to the reader:  that it makes sense to periodize certain almost complex structures and to push them forward/pull them  back  through a power map $(z,w) \mapsto (z^k,w)$. These verifications are simpler if one passes to polar coordinates by pulling structures back through  $(x,y,w) \mapsto (xe^{iy}, w): \R^{2n+2}_+ \to \C^{n+1}$ and applying arguments similar to those appearing  in the  proofs of Theorems~\ref{2dexist} and \ref{exist}. We omit these details below to avoid obfuscation of the structure of the argument by an excessive number of coordinate changes.

\begin{propn}\label{sector}
Let $\V$ be an almost complex structure on $\C^{n+1}$ that agrees with the standard complex structure  to infinite order along $\{0\} \times \C^n$ in the sense that $\V$ admits a global frame of the form $\partial_{\overline z_j} +  \Gamma_j$, $j=0,\ldots,n$ where $\Gamma_j$ is a smooth complex vector field vanishing to infinite order on $\{0\}\times \C^{n}$.  Suppose further that there is a sector  $\Omega_\theta = \{(z,w) : |\Arg(z)|< \theta\}$ with $0 < \theta < \frac{\pi}{2}$ such that $\V$ is integrable on $\Omega_\theta$ near the origin. Then there exists an open neighbourhood $U \subseteq \C^{n+1}$ of the origin on which is defined a complex structure that agrees $\V$ on $U \cap \Omega_{\theta/2}$, i.e. on the sector of half the size. 
\end{propn}
\begin{proof}
Let $k \geq 2$ be the least positive integer for which $k\theta \geq  \frac{\pi}{2}$. Consider the transformations of $\C^{n+1}$ which  apply the $k$th power map and a rotation by $2\pi/k$ to  the first coordinate:
\begin{align*}
P_k : \C^{n+1}\to\C^{n+1} &&  P_k(z,w) &= (z^k, w) \\
R_k:\C^{n+1} \to \C^{n+1} && R_k(z,w) &= (e^{2\pi i /k} z, w).
\end{align*}

Since $\frac{k}{2} \theta \leq (k-1) \theta <\frac{\pi}{2}$, we have 
$\theta  < \frac{\pi}{k}$. In other words, the sector $\Omega_\theta$ is strictly contained in the sector $\Omega_{\pi/k}$ which is a fundamental domain for the rotation $R_k$. From this containment and a  simple bump function argument in polar coordinates, we may construct another almost complex structure $\V_1$ on $\C^{n+1}$ such that: 
\begin{enumerate}[(i)]
\item $\V_1$ and $\V$  coincide on  $\Omega_\theta$,
\item $\V_1$ still agrees with the standard complex structure to infinite order along $\{0\}\times\C^n$,
\item and $\V_1$ is  $R_k$-invariant, meaning $(R_k)_*(\V_1) =\V_1$.
\end{enumerate}
Using (ii) and (iii),  it can be checked that the pushforward $\V_2 \coloneqq (P_k)_*(\V_1)$ makes sense and is another almost complex structure on $\C^{n+1}$ agreeing with the standard complex structure to infinite order along $\{0\}\times\C^n$. 

By construction, $\V_2$ is integrable  on $P_k(\Omega_\theta) = \Omega_{k\theta}$ near the origin. Since $k \theta \geq \frac{\pi}{2}$ and the special domain $K$ in Proposition~\ref{specialdomain} is contained in the  half space $\R^{2n+2}_+$, we have that the restriction of  $\V_2$ to  $K$ is involutive near the origin. We may now apply Proposition~\ref{specialdomain} to obtain a   neighbourhood $U_3 \subseteq \C^{n+1}$ on which is defined a complex structure $\V_3$ that agrees with $\V_2$ on $U_3 \cap K$. Note $\frac{k}{2}\theta < \frac{\pi}{2}$, so the domain $K$ contains all points of $\Omega_{k \theta/2}$ sufficiently close to the origin. Thus, the complex structure $\V_3$ in particular agrees with $\V_2$ near the origin on $\Omega_{k \theta/2}$.

Note that $\V_3$ necessarily also agrees with the standard complex structure to infinite order along $\{0\}\times\C^{n+1}$. It  follows that  the pullback $\V_4 = (P_k)^*(\V_3)$ is a well-defined complex structure on $U_4 \coloneqq P_k^{-1}(U_3)$, still agreeing with the standard complex structure to infinite order along $\{0\}\times\C^n$.  By design, $\V_4$ agrees with $\V$ near the origin of $U_4 \cap \Omega_{\theta/2}$. Possibly replacing $U_4$ by a smaller neighbourhood $U$ of the origin, we obtain that the complex structure $\V_4$ agrees with the original almost complex structure $\V$ on  $U \cap \Omega_{\theta/2}$.
\end{proof}

\section{Existence of \texorpdfstring{$b$}{b}-holomorphic functions: general case}\label{sec:exist}

In this section, we extend the results of Section~\ref{sec:2exist} to the general case. That is, we show that  nontrivial $b$-holomorphic functions exist near points on the boundary of any complex $b$-manifold. By Theorem~\ref{Mendoza}, this amounts to the following.

\begin{thm}\label{exist}
Let $M \subseteq  \R^{2n+2}_+$ be a relative neighbourhood of the origin equipped with a complex $b$-structure ${\bT^{0,1}}M$ spanned by:
\begin{align*}
L_0 &\coloneqq \bdel_{\overline z_0} + \Gamma_0 \\
L_j &\coloneqq \partial_{\overline z_j} + \Gamma_j && j=1,\ldots,n
\end{align*} 
where   $\Gamma_0,\ldots,\Gamma_n$ are smooth complex vector fields on $M$ that  vanish to infinite order on $\{0\}\times\R^{2n+1}$. Then there exists a neigbourhood $U \subseteq M$ of the origin admitting $b$-holomorphic functions $f_0,f_1,\ldots,f_n : U \to \C$ such that:
\begin{enumerate}[(i)]
\item $f_0$ vanishes on $\{0\}\times\R^{2n+1}$  and $f_j(0)=0$ for $j=1,\ldots,n$,
\item $\partial_{x_j} f_k (0)=\delta_{j,k}$ for  $j,k \in \{0,1,\ldots,n\}$, using Kronecker delta notation.
\end{enumerate}
\end{thm}
As noted in Remark~\ref{techniquesneeded},  the higher-dimensional cases  will require an appeal to the results of  Section~\ref{new} to account for the loss of automatic integrability. Otherwise, the  proof will proceed  largely as it did in the two-dimensional case (Theorem~\ref{2dexist}), so we will be somewhat brief. First, we state  straightforward higher-dimensional generalizations of Lemmas~\ref{pushlem} and \ref{divzbar}. Again, we omit the proofs of these statements.

\begin{lemma}\label{genpushlem}
Suppose $\gamma : \R^{2n+2}_+\to \C$ is a smooth function that   vanishes to infinite order on the real hyperplane $Z_r \coloneqq \{0\}\times\R^{2n+1}$ and is  $2\pi$-periodic in $y_0$. Define
\begin{align*}
g &: \R^2_+ \to \C &  
g(x,y)&=xe^{iy} &&& 
\widetilde g &: \R^{2n+2}_+ \to \C^{n+1} &  
\widetilde g = g \times \mathrm{id}
\end{align*}
 Then, the pushforward $\widetilde g_*(\gamma)$ is a well-defined smooth function on $\C^{n+1}$ that vanishes to infinite order on the complex hyperplane $Z_c \coloneqq \{0\}\times \C^n$. \qed
\end{lemma}

\begin{cor}\label{vecpush}
Suppose that $\Gamma$ is a smooth complex vector field on  $\R^{2n+2}_+$ vanishing to infinite order on $\{0\}\times\R^{2n+1}$. Then, the pushforward $\widetilde g_*(\Gamma)$ is a well-defined smooth complex vector field on $\C^{n+1}$ vanishing to infinite order on $\{0\}\times\C^n$.
\end{cor}
\begin{proof}
We may write  $\Gamma = \gamma_0 \cdot \bdel_{z_0}  + \gamma_0' \cdot \bdel_{\overline z_0} + \sum_{j=1}^n \gamma_j \partial_{z_j} + \gamma_j' \partial_{\overline z_j}$ where $\gamma_j, \gamma_j'$ are smooth functions on $\R^{2n+2}_+$ vanishing to infinite order on $\{0\}\times\R^{2n+1}$ that are $2\pi$-periodic in $y_0$. Then, from Lemma~\ref{pushforwards} and Lemma~\ref{genpushlem}, we have 
\[ \widetilde g_*(\Gamma)=\widetilde g_*(\gamma_0) z_0 \partial_{z_0} + \widetilde g_*(\gamma_0') \overline z_0 \partial_{\overline z_0} + \sum_{j=1}^n \widetilde g_*(\gamma_j) \partial_{z_j} + \widetilde g_*(\gamma_j') \partial_{\overline z_j}.\] 
\end{proof}

\begin{lemma}\label{gendivzbar}
Suppose $\gamma$ is a smooth complex-valued function on $\C^{n+1}$ vanishing to infinite order on $Z_c \coloneqq \{0\} \times \C^n$. Then, $(1/z_0)\gamma$ and $(1/\overline z_0) \gamma$ extend to smooth complex-valued functions on $\C^{n+1}$ vanishing to infinite order on  $Z_c $. \qed

\end{lemma}

Having made the above preparations, we now proceed to the proof of this section's main result.

\begin{proof}[Proof of Theorem~\ref{exist}]
As in the proof of Theorem~\ref{2dexist}, since  local $b$-holomorphic functions are all we desire, we assume the deformation terms $\Gamma_0,\ldots,\Gamma_n$ are defined on all of $\R^{2n+2}_+$ and $2\pi$-periodic in $y_0$. To achieve this, we use  bump functions to construct   global periodic vector fields agreeing with the given ones near the origin. However, unlike in Theorem~\ref{2dexist}, we cannot ensure that the subbundle framed by the extended $L_0,\ldots,L_n$ is involutive because the use of bump functions  may interfere with that condition.  In other words, we can only assume without loss of generality that $M=\R^{2n+2}_+$ if we also  accept that, away from the origin, $L_0,\ldots,L_n$ only define what it would be  natural to call an ``almost complex $b$-structure'': the splitting $\C{\bT M} = {\bT^{1,0} M} \oplus {\bT^{0,1} M}$ holds everywhere, but $\bT^{0,1} M$ may only be involutive on a neighbourhood of the origin. 

Regardless, Lemma~\ref{pushforwards} and Corollary~\ref{vecpush} together imply that the  vector fields $L_j$ and $\Gamma_j$ have   well-defined pushforwards by $\widetilde g$. Indeed,
\begin{align*}
\widetilde g_*(L_0) &= \overline z_0 \partial_{\overline z_0} + \widetilde g_*(\Gamma_0) \\
\widetilde g_*(L_j) &= \partial_{\overline z_j} + \widetilde g_*(\Gamma_j) && j=1,\ldots,n, 
\end{align*}
where  $\widetilde g_*(\Gamma_j)$ is a smooth vector field on $\C^{n+1}$ vanishing to infinite order on $Z_c \coloneqq \{0\} \times \C^n$ for $j=0,\ldots,n$.  Because $\widetilde g$ is a local diffeomorphism $\R^{2n+2} \setminus Z_r \to \C^{n+1} \setminus Z_c $, we have that  $\widetilde g_*(L_j)$, $j=0,\ldots,n$ at least define an almost complex structure on $\C^{n+1}\setminus Z_c$.   Going further, we can use  Lemma~\ref{gendivzbar} to write $\widetilde g_*( \Gamma_0) = \overline z_0 \Gamma_0'$, where $\Gamma_0'$ is another smooth vector field on $\C^{n+1}$ vanishing to infinite order on $Z_c $. We claim that
\begin{align}\label{newbasis}
\tfrac{1}{\overline z_0} \widetilde g_*(L_0) 
&= \partial_{\overline z_0} + \Gamma_0'  && \\
\widetilde g_*(L_j) 
&=\partial_{\overline z_j}  + \widetilde g_*(\Gamma_j) && j=1,\ldots,n \nonumber
\end{align}
define an  almost complex structure  $\V$ on the whole of $\C^{n+1}$ (extending the aforementioned almost complex structure on  $\C^{n+1} \setminus Z_c $). Indeed, the deformation terms vanish on $Z_c $, so the vector fields \eqref{newbasis}, together with their complex conjugates, form a global frame for the complexified tangent bundle.

Manifestly, $\V$ agrees with the standard complex structure on $\C^{n+1}$ to infinite order along $Z_c$. Furthermore, because  $\bT^{0,1} M$ is involutive near the origin, it follows that its pushforward $\V$ by the polar coordinate change $\widetilde g$ is likewise involutive near the origin of some sector  $\Omega_\theta = \{ ( z_0,z_1,\ldots,z_n) : |\Arg(z_0)|<\theta\}$ where $0 < \theta<\frac{\pi}{2}$. Thus, we may apply Proposition~\ref{sector} and obtain an open neighbourhood $U_1 \subseteq \C^{n+1}$ of the origin on which is defined a complex structure $\V_1$ agreeing with $\V$ on $U_1 \cap \Omega_{\theta/2}$.

By an application of the classical Newlander-Nirenberg theorem, there is an open neighbourhood $U_2 \subseteq U_1$ of $0$ and smooth functions $h_0,\ldots,h_n : U_2 \to \C$ that are holomorphic for the complex structure $\V_1$  and furthermore satisfy $h_j(0)=0$, $\partial_{x_j} h_k(0)= \delta_{j,k}$ for $j,k \in \{0,1,\ldots,n\}$. In fact, again using the vanishing of the deformation terms along $Z_c$, the complex structure $\V_1$  on $U_1$  is such that $Z_c  \cap U_1$ sits as a complex submanifold. Indeed, the inherited complex structure  on $Z_c $ is its standard one spanned by $\partial_{\overline z_1},\ldots,\partial_{\overline z_n}$. Using the implicit function theorem for several complex variables, complex hypersurfaces can locally be written as preimages of regular values of holomorphic functions. Thus, possibly shrinking $U_2$ around the origin, we may also choose $h_0$ to vanish on $Z_c $. Lastly, recalling that the complex structure $\V_1$ agrees with $\V$ on $U_1 \cap \Omega_{\theta/2}$, we put 
\begin{align*}
U \coloneqq (\widetilde g \,)^{-1}(U_2 \cap \Omega_{\theta/2}) && f_j \coloneqq h_j \circ \widetilde g && j=0,\ldots,n
\end{align*}
to complete the proof.
\end{proof}

\section{The \texorpdfstring{$b$}{b}-Newlander-Nirenberg theorem: general case}\label{sec:main}

This section is devoted to the proof of our main result Theorem~\ref{mainthm}. In light of Mendoza's Theorem~\ref{Mendoza}, we may begin on relative neighbourhood of the origin $M \subseteq  \R^{2n+2}_+$   equipped with a complex $b$-structure ${\bT^{0,1}}M$ spanned by: 
\begin{align*}
L_0 &\coloneqq \bdel_{\overline z} + \Gamma_0 \\
L_j &\coloneqq \partial_{\overline z_j} + \Gamma_j && j=1,\ldots,n
\end{align*} 
where   $\Gamma_0,\ldots,\Gamma_n$ are smooth  complex vector fields on $M$ that vanish to infinite order on $Z_r = \{0\}\times\R^{2n+1}$. 

By Theorem~\ref{exist}, there is an open set $U \subseteq M$ containing $0$ and smooth functions $f_0,\ldots,f_n : U \to \C$ satisfying:
\begin{enumerate}[(i)]
\item $L_j f_k = 0$ for $j,k\in\{0,1,\ldots,n\}$ ($b$-holomorphicity),
\item $f_0$ vanishes on $Z_r = \{0\}\times\R^{2n+1}$ and $f_j(0)=0$ for $j=1,\ldots,n$,
\item $\partial_{x_j} f_k(0)= \delta_{j,k}$ for $j,k \in \{0,1,\ldots,n\}$.
\end{enumerate}
As in Section~\ref{sec:2main}, we decompose $f_0$ into real and imaginary parts $f_0=u+iv$. Since $u$ and $v$ vanish on $Z_r$, we may write $u=x_0 \cdot a$ and $v=x_0 \cdot b$, where $a$ and $b$ are real-valued smooth functions. As in Section~\ref{sec:2main}, by shrinking $U$, we may enforce that $a$ is nowhere vanishing on $U$ and that: 
\begin{enumerate}[(a)]
\item $F= (u, \tfrac{b}{a}, f_1,\ldots,f_n)$ is a diffeomorphism from $U$ onto an open set $F(U) \subseteq \R^{2n+2}_+$,
\item $F(0)=0$ and $F(U\cap Z_r) = F(U)\cap Z_r$,
\item $f = \widetilde \kappa \circ F$ where $f=(f_0,f_1,\ldots,f_n)$ and $\widetilde \kappa = \kappa \times \mathrm{id} : \R^{2n+2}\to\R^{2n+2}$ where $\kappa:\R^2 \to \R^2$ is $\kappa(x,y)=(x,xy)$.
\end{enumerate}
As in Section~\ref{sec:2main}, we argue that the pushforward by $F$ of the complex $b$-structure on $U$ spanned by $L_0,\ldots,L_n$ is independent of the deformation terms $\Gamma_0,\ldots,\Gamma_n$. Again, by Proposition~\ref{restriso},  it suffices check this away from the boundary $\{0\}\times\R^{2n+1}$. Using the fact that  $\widetilde \kappa$ restricts to a self-diffeomorphism of $\R^{2n+2} \setminus Z_r$, we have that $f$ defines a diffeomorphism of $U \setminus Z_r$ onto $\kappa(F(U\setminus Z_r))$. Because the components of $f$ are holomorphic for the induced complex structure on $U \setminus Z_r$, Proposition~\ref{holocharts} implies that $f|_{U\setminus Z_r}$ pushes forward the subbundle  spanned by $L_0,\ldots,L_n$ to the subbundle spanned by $\partial_{\overline z_0},\ldots,\partial_{\overline z_n}$. Thus, the pushforward by $F|_{U\setminus Z_r}$ of the subbundle spanned by $L_0,\ldots,L_n$  is completely determined to be the pullback of the standard complex structure on $\kappa(F(U\setminus Z_r))$ by $\kappa$. Running the same argument in the case where the deformation terms are all zero, we obtain the desired coordinate change.

\section{Remarks on singular pushforwards}\label{generalpush}

The polar coordinate change $g:\R^2 \to \C$,  $g(x,y)=xe^{iy}$
played a crucial role in this article by  enabling us to relate the nonelliptic operator $\bdel_z = \frac{1}{2}(x\partial_x-i\partial_y)$ to the standard complex partial derivative operator $\partial_z = \frac{1}{2}(\partial_x-i\partial_y)$ by way of the elementary, but perhaps somewhat mysterious, pushforward formula $g_*( \bdel_{z}) = z \partial_z$. In this section, we shed additional light on this pushforward formula by placing it in a more general context. The material in this section  will also play a role in planned future work on the Newlander-Nirenberg theorem for complex $b^k$-manifolds \cite{Barron-Francis}.

Suppose that, more generally, we wish to relate the  vector field
\begin{align*}
\tfrac{1}{2}( x^k \partial_x  - i \partial_y ) 
\end{align*}
to $\partial_z$, for $k$ a positive integer. More generally still,  suppose we wish to relate 
\begin{align*}
\tfrac{1}{2}( f(x) \partial_x - i \partial_y)
\end{align*}
to $\partial_z$, where $f(z)$ is an entire function on $\C$ whose restriction $f(x)$ to $\R$ is real-valued.

We first recall a few more-or-less notational points concerning holomorphic vector fields and imaginary-time flows. One may refer to \cite{Forstneric}, p.~39 for additional details. Let $f$ be a holomorphic function on $\C$ so that $V \coloneqq  f(z)\partial_z$ is a holomorphic vector field on $\C$. We  may decompose $V$ as 
\[ V = \tfrac{1}{2}(X-iJX), \]
where the real vector fields $
X =\mathrm{Re}(f)\partial_x + \mathrm{Im}(f)\partial_y$ and $ JX = -\mathrm{Im}(f) \partial_x + \mathrm{Re}(f)\partial_y$ commute with one another. We may then define the complex-time  flow of $V$ by\begin{align*}
\phi^V_w(z) \coloneqq \phi^X_s \circ \phi^{JX}_t (z) && \text{where }w=s+it,
\end{align*}
with the usual caveat that the flow need only be defined for $w$ sufficiently close to $0$.

This flow is jointly holomorphic in $w$ and $z$. One practical consequence of this joint holomorphicity is that, if $f(z)$ is real on the $x$-axis, so that $X$ coincides with  $f(x) \partial_x$ on the $x$-axis, then the complex-time flow of the holomorphic vector field $f(z)\partial_z$ may be obtained by analytic continuation in both variables of the real-time flow of the one-dimensional vector field $f(x)\partial_x$. 

\begin{ex}
The real-time flow of the one-dimensional vector field $X_1 \coloneqq  x \partial_x$ is given by $\phi^{X_1}_t(x) = e^tx$. Accordingly, the complex-time flow of the holomorphic vector field $V_1 \coloneqq z \partial_z$ is given by $\phi^{V_1}_w(z) = e^wz$. 
\end{ex}
\begin{ex}
The real-time flow of the one-dimensional vector field $X_2 \coloneqq x^2 \partial_x$ is given by $\phi^{X_2}_t(x) = \frac{x}{1-tx}$.  Accordingly, the complex-time flow of the holomorphic vector field $V_2 \coloneqq z^2 \partial_z$ is given by $\phi^{V_2}_w(z) = \frac{z}{1-wz}$.
\end{ex}
\begin{ex}
For any integer $k \geq 2$, the real-time flow of the one-dimensional vector field $X_k \coloneqq x^k \partial_x$ is given by $\phi^{X_k}_t(x) = \frac{x}{\sqrt[k-1]{1-(k-1)tx^{k-1}}}$. Accordingly, the complex-time flow of the holomorphic vector field $V_k \coloneqq z^k \partial_z$  is given by $\phi^{V_k}_w(z) = \frac{z}{\sqrt[k-1]{1-(k-1)wz^{k-1}}}$.
\end{ex}

Bearing in mind the above notations, we now give the main result of this section.

\begin{propn}
Let $f$ be a holomorphic function on $\C$ that is real-valued on $\R$. Let $U \subseteq \R^2$ be a sufficiently small  open neighbourhood of  $\R \times \{0\}$ that $h:U \to \C$, $h(x,y) = \phi^{f(z)\partial_z}_{iy}(x)$ is defined. Then, the vector field  $\frac{1}{2}(f(x)\partial_x - i \partial_y)$ on $U$ is $h$-related to the holomorphic vector field $f(z)\partial_z$ on $\C$.  
\end{propn}
\begin{proof}
Decompose $V=f(z)\partial_z$ in the form $V = \frac{1}{2}(X-iJX)$, as above. Thus, by definition, $h(x,y)=\phi^{JX}_y(x)$. For fixed $(x,y)\in U$, taking $t$ sufficiently small, we have
\[ h(x,y+t) = \phi^{JX}_{y+t}(x) = \phi^{JX}_t (h(x,y)). \]
Thus, the vector field $\partial_y$ is $h$-related to $JX$. Similarly, 
\[ \phi^X_t h(x,y) = \phi^X_t  \phi^{JX}_y(x) = \phi^{JX}_y \phi^X_t(x) =  \phi^{JX}_y \phi^{f(x)\partial_x}_t(x) = h( \phi^{f(x)\partial_x}_t(x),y) = h(\phi^{f(x)\partial_x}_t(x,y)). \]
Thus, the vector field $f(x) \partial_x$ is $h$-related to $X$. The result follows by linearity.
\end{proof}

\begin{cor}
Let $g_2:\R^2 \to \C$, $g_2(x,y)=\frac{x}{1-ixy}$. Then, $\frac{1}{2}(x^2\partial_x-i\partial_y)$ is $g_2$-related to $z^2 \partial_z$. \end{cor}

\begin{cor}
Fix an integer $k \geq 2$. Let $g_k:\R^2 \to \C$,  $g_k(x,y)=\frac{x}{\sqrt[k-1]{1-i(k-1)x^{k-1}y}}$. Then  $\tfrac{1}{2}(x^k\partial_x - i \partial_y )$ is $g_k$-related to $z^k \partial_z$.
\end{cor}

We remark that $g_k(x,y)=\frac{x}{\sqrt[k-1]{1-i(k-1)x^{k-1}y}}$ collapses $Z_r=\{0\} \times \R$ to $\{0\}$ and defines a diffeomorphism $\R^2 \setminus Z_r \to \{ re^{i\theta} \in \C : r \neq 0, -\frac{\pi}{2(k-1)} < \theta <\frac{\pi}{2(k-1)}\}$. Note that the line $\mathrm{Re}(z)=1$ is completely contained in the principal branch of $\sqrt[k-1]{z}$.  For purposes of clarification, sample plots of the real vector field $JX$ in the decomposition $z^k \partial_z = \frac{1}{2}(X-iJX)$ are given below. Note all integral curves passing through the $x$-axis are complete.

\begin{figure}
    \centering
    \includegraphics[width=0.3\linewidth]{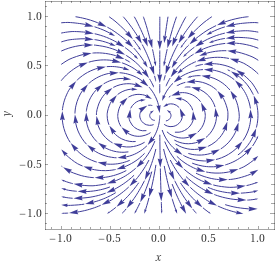}
    \includegraphics[width=0.3\linewidth]{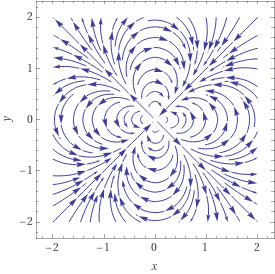}
    \includegraphics[width=0.3\linewidth]{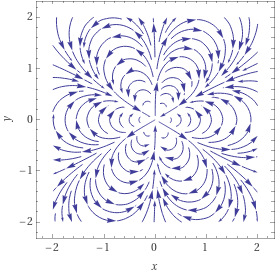}
    \caption{The vector field $JX = -\mathrm{Im}(z^k) \partial_x + i \mathrm{Re}(z^k) \partial_y$ for $k=2,3,4$.}
    \label{fig:enter-label}
\end{figure}

\bibliographystyle{abbrv}
\bibliography{V11bib}

\end{document}